\documentclass[11pt,twoside,reqno]{amsart}
\usepackage{amsmath, amssymb, amsthm, color}
\usepackage[colorlinks=false,citecolor=blue]{hyperref}
\usepackage[tmargin=1in,bmargin=1in,rmargin=1in,lmargin=1in]{geometry}

\newtheorem{theorem}{Theorem}[section]
\newtheorem{lemma}[theorem]{Lemma}
\newtheorem{prop}[theorem]{Proposition}
\newtheorem{cor}[theorem]{Corollary}
\theoremstyle{definition}
\newtheorem{definition}[theorem]{Definition}
\newtheorem{example}[theorem]{Example}
\theoremstyle{remark}
\newtheorem{remark}[theorem]{\bf{Remark}}
\numberwithin{equation}{section}
\allowdisplaybreaks

\begin{document}

\title[Strengthening of spectral radius, numerical radius, and Berezin radius inequalities]
{Strengthening of spectral radius, numerical radius, and Berezin radius inequalities}

\author{Pintu Bhunia}
\address[P.~Bhunia]{Department of Mathematics, Indian Institute of
Science, Bangalore 560012, India}
\email{\tt pintubhunia5206@gmail.com; pintubhunia@iisc.ac.in}


\thanks{The author would like to thank Prof. Apoorva Khare for useful discussions.
The author was supported by National Post-Doctoral Fellowship PDF/2022/000325 from SERB (Govt.\ of India) and SwarnaJayanti Fellowship SB/SJF/2019-20/14 (PI: Apoorva Khare) from SERB (Govt.\ of India).
The author also would like to thank National Board of Higher Mathematics (Govt.\ of India) for the financial support in the form of NBHM Post-Doctoral Fellowship 0204/16(3)/2024/R\&D-II/6747.
}





\date{\today}
\maketitle

\begin{abstract}
Suppose $\mathcal{H}_1, \mathcal{H}_2, \ldots, \mathcal{H}_n$ are arbitrary complex Hilbert spaces, and  ${\bf A}=[A_{ij}]$ is an $n\times n$ operator matrix with  $A_{ij}\in \mathcal{B}(\mathcal{H}_j, \mathcal{H}_i).$ We show that 
$w({\bf A}) \leq w\left(\begin{bmatrix}
		 a_{ij}
	\end{bmatrix}_{i,j=1}^n  \right),$
where $w(\cdot)$ denotes the numerical radius and the entries 
$$ a_{ij}=\begin{cases}
    w(A_{ii}) & \textit{if $i=j$},\\
    \sqrt{ \left( \|A_{ij}\|+\|A_{ji}\| \right)^2- \left(\|A_{ij}\| \|A_{ji}\|-w(A_{ji}A_{ij}) \right)}^{} & \textit{if $i<j$},\\
    0  & \textit{if $i>j$.}
\end{cases}$$
This bound improves $w({\bf A}) \leq w\left(\begin{bmatrix}
		 a'_{ij}
	\end{bmatrix}_{i,j=1}^n  \right),$ where $a'_{ij}=w(A_{ii})$ if $i=j$ and $a'_{ij}=\|A_{ij}\|$ if $i\neq j$ [Abu-Omar and Kittaneh, Linear Algebra Appl. 2015]. 
   We deduce an upper bound for the Kronecker products $A\otimes B$, where $A\in \mathcal{M}_n(\mathbb{C})$ and $B\in \mathcal{B}(\mathcal{H}_1)$, which refines Holbrook's classical bound $w(A\otimes B)\leq w(A)\|B\|$ [J. Reine Angew. Math. 1969], when all entries of $A$ are non-negative.  Further, we obtain the Berezin radius inequalities for $n\times n$ operator matrices where the entries are reproducing kernel Hilbert space operators, which refine Bakherad's inequalities [Czech. Math. J. 2018]. We provide
an example, which illustrates these inequalities for some concrete operators on the Hardy--Hilbert space.
Applying the numerical radius bounds, we show that if $A_i \in \mathcal{B}(\mathcal{H}_i, \mathcal{H}_1) $ and $B_i\in \mathcal{B}(\mathcal{H}_1, \mathcal{H}_i)$ for $i=1,2,$ then 
    \begin{eqnarray*}
        r(A_1B_1+A_2B_2) \leq   \frac{ 1 }{2 } \left(w(B_1A_1)+w(B_2A_2) \right) 
         +  \frac{ 1 }{2 }  \sqrt{ \left(w(B_1A_1)-w(B_2A_2)\right)^2 + 3\|B_1A_2\|\|B_2A_1\| + \eta},
    \end{eqnarray*}
where $\eta=w(B_2A_1 B_1A_2)$, and $r(\cdot)$ denotes the spectral radius. This refines the inequalities in [Kittaneh, Proc. Amer. Math. Soc. 2006] and [Abu-Omar and Kittaneh, Studia Math. 2013]. We also achieve a bound for the roots of an algebraic equation, which is sharper than Abdurakhmanov's bound [Math. USSR-Sb. (1988)].
    \end{abstract}


\section{Introduction and Main results}

Throughout this work, $\mathcal{B}(\mathcal{H})$ denotes the $C^*$-algebra of all bounded linear operators on a complex Hilbert space $ \mathcal{H}$ with inner product $\langle.,.\rangle.$
The purpose of this work is to study the numerical radius inequalities for $n\times n$ operator matrices, i.e., matrices whose entries are in $\mathcal{B}(\mathcal{H})$. Using similar approaches, we also study the Berezin radius (Berezin number) inequalities for reproducing kernel Hilbert space operators. 
We apply these inequalities to study the spectral radius inequalities for the sums, products, and commutators of bounded linear operators and to estimate the roots of an algebraic equation.

Calculating the exact value of the spectral radius of a complex matrix is not always possible. The problem becomes more challenging for operator matrices. This problem, which has a wide range of applications, has stimulated mathematicians to establish estimates for the spectral radius of operator matrices by developing spectral radius inequalities. The numerical radius is one of the nice concepts for estimating the spectral radius of a bounded linear operator.
The study of various kinds of numerical range and numerical radius
of a bounded linear operator goes back
at least to Hausdorff \cite{Hausdorff} and Toeplitz \cite{Toeplitz}  (see
also \cite{Halmos}). In fact, even earlier, it was started from the Rayleigh
quotients in the 19th century. The numerical radius has
seen widespread usage through applications in numerical analysis, functional analysis,
operator theory, systems theory, quantum information theory and quantum computing. We refer the reader to e.g.~\cite{book}
for more on this.

\begin{definition}
The numerical range $W(A)$ of a bounded linear operator $A\in
\mathcal{B}(\mathcal{H})$ is defined as $W(A) := \left\{ \langle
Ax,x\rangle: x\in \mathcal{H}, \|x\|=1  \right\}$, and the
numerical radius $w(A)$ is defined as
$w(A) := \sup\left\{ |\langle Ax,x\rangle|: x\in \mathcal{H}, \|x\|=1
\right\} = \sup\left\{ |\lambda|: \lambda\in W(A)  \right\}.$
\end{definition}

It is well known that the numerical radius defines a norm on
$\mathcal{B}(\mathcal{H})$ and is equivalent to the operator norm
$\|A\|=\sup\left\{ \| Ax\|: x\in \mathcal{H}, \|x\|=1  \right\}$ via the
relation 
\begin{eqnarray}\label{E-V}
    \frac{1}{2}\|A\| \leq w(A) \leq \|A\|.
\end{eqnarray}
These inequalities are sharp, $\frac{1}{2}\|A\|= w(A) $ if $A^2=0$ and $ w(A) = \|A\|$ if $A$ is normal.
It is also a basic fact that the norm $w(\cdot)$ is weakly unitarily invariant (i.e.,  $w(U^{*}AU) = w(A)$ for all $A\in \mathcal{B}(\mathcal{H})$ and for all unitary $U\in \mathcal{B}(\mathcal{H})$. 
In addition to that, one important inequality for the numerical radius is $w(A^n) \leq w^n(A)$ for all $n\in \mathbb{N}.$

Over the years many researchers have been trying to improve the numerical radius inequalities \eqref{E-V} (we refer to e.g. the books \cite{book} and \cite{Wu}).

Numerical radius inequalities for operator matrices have also been well studied, see e.g. \cite{Abu_LAA_2015}, \cite{Bhunia-AdM24}, \cite{Bag}, \cite{Paul}, \cite{Hiz2011} and \cite{Hou_1995}. Let $\mathcal{H}_1,\mathcal{H}_2,\ldots,\mathcal{H}_n$ be complex Hilbert spaces, and let ${\bf A}=[A_{ij}]$ be an $n \times n$ operator matrix with entries  $A_{ij}\in \mathcal{B}(\mathcal{H}_j, \mathcal{H}_i),$ the space of all bounded linear operators from $\mathcal{H}_j$ to $ \mathcal{H}_i$.
    Then ${\bf A}\in \mathcal{B}( {\oplus}_{i=1}^n\mathcal{H}_i)$ is defined as ${\bf A}x:=  \begin{pmatrix}
	\sum_{i=1}^n A_{1i}x_i,
	\sum_{i=1}^n A_{2i}x_i,
	\ldots,
	\sum_{i=1}^n A_{ni}x_i
\end{pmatrix}^T$ for all $x=  \begin{pmatrix}
	 x_1,
	x_2,
	\ldots,
	x_n
\end{pmatrix}^T \in {\oplus}_{i=1}^n\mathcal{H}_i$.

Operator matrices provide a useful tool to study Hilbert space operators, which have been extensively studied, see e.g. \cite{Halmos}.
In \cite{Hou_1995},  Hou and Du provided a useful bound for the numerical radius of an 
 $n \times n$ operator matrix ${\bf A}=[A_{ij}]$, where $A_{ij}\in \mathcal{B}(\mathcal{H}_j, \mathcal{H}_i)$.
 They proved that
 \begin{eqnarray}\label{p7}
	w({\bf A}) \leq 	w\left(\begin{bmatrix}
	\|A_{ij}\| 
\end{bmatrix}_{i,j=1}^n  \right).
\end{eqnarray}
Further, in \cite{Abu_LAA_2015}, Abu-Omar and Kittaneh provided a considerable improvement of \eqref{p7}, namely,
\begin{eqnarray}\label{p8}
	w({\bf A}) \leq 	w\left(\begin{bmatrix}
		a_{ij} 
	\end{bmatrix}_{i,j=1}^n  \right) =w\left(\begin{bmatrix}
		b_{ij} 
	\end{bmatrix}_{i,j=1}^n  \right), 
\end{eqnarray}
 \textit{where the entries }
$ a_{ij}=\begin{cases}
	    w(A_{ii}) & \textit{ if } i=j,\\
         \|A_{ij}\| & \textit{ if } i\neq j\\
	\end{cases} \textit{ and }\, b_{ij}=\begin{cases}
	    w(A_{ii}) & \textit{ if } i=j,\\
         \|A_{ij}\|+\|A_{ji}\| & \textit{ if } i< j,\\
         0 & \textit{ if } i> j.
	\end{cases}$
    
\noindent The equality in \eqref{p8} is obtained using the following well-known fact (see \cite[p. 44]{Horn} and \cite{PMSC}):
If $A=[a_{ij}]\in \mathcal{M}_n(\mathbb{C})$ with all entries $a_{ij}\geq 0$, then 
\begin{eqnarray}\label{p9}
	w(A)=\frac{1}{2}w\left( [a_{ij}+a_{ji}]_{i,j=1}^n \right)=\frac{1}{2}r\left( [a_{ij}+a_{ji}]_{i,j=1}^n \right),
\end{eqnarray}
 where $r(\cdot)$ denotes the spectral radius.  Recall that for $A \in \mathcal{B}(\mathcal{H})$,  $r(A)= \sup\{|\lambda|: \lambda \in \sigma (A) \}$, where $\sigma (A)$ denotes the spectrum of $A$. 
 
  Recently, Bhunia \cite{Bhunia-AdM24} developed an improvement of \eqref{p8} for the case $\mathcal{H}_1= \mathcal{H}_2= \ldots= \mathcal{H}_n$. In particular, he showed that $w({\bf A}) \leq w\left(\begin{bmatrix}
		a_{ij} 
	\end{bmatrix}_{i,j=1}^n  \right)$, where the entries $a_{ij}=w(A_{ii})$ if $i=j$, $a_{ij}= \sqrt{\left\| |A_{ij}|+ |A^*_{ji}| \right\|^{}   \left\| |A^*_{ij}|+ |A_{ji}| \right\|^{}}$ if $i<j$ and $a_{ij}=0$ if $i>j.$ Here $|A|$ denotes the positive square root of $A^*A,$ i.e., $|A|=\sqrt{A^*A}.$

 \smallskip
We begin by refining Abu-Omar and Kittaneh's bound \eqref{p8}. Here is the main result of this work.

\begin{theorem}\label{thm2}
Let $\mathcal{H}_1, \mathcal{H}_2, \ldots, \mathcal{H}_n$ be complex Hilbert spaces, and let ${\bf A}=[A_{ij}]$ be an $n\times n$ operator matrix with  $A_{ij}\in \mathcal{B}(\mathcal{H}_j, \mathcal{H}_i).$ Then 
\begin{eqnarray*}
    w({\bf A}) \leq w\left(\begin{bmatrix}
		a_{ij} 
	\end{bmatrix}_{i,j=1}^n  \right), \textit{ where }   a_{ij}= \begin{cases}
	   w(A_{ii}) & \textit{if } i=j,\\
        \sqrt{ \left( \|A_{ij}\|+\|A_{ji}\| \right)^2- \left(\|A_{ij}\| \|A_{ji}\|-w(A_{ji}A_{ij}) \right)}  & \textit{if } i<j,\\
        0  & \textit{if } i>j.
	\end{cases}
\end{eqnarray*} 
\end{theorem}

\begin{remark}\label{rem1}
Here we show that the bound in Theorem \ref{thm2} is at most the bound in \eqref{p8}.  
The relation \eqref{p9} and the spectral radius monotonicity of matrices with non-negative entries imply (see \cite[p. 491]{Horn2}) that if $A=\begin{bmatrix}
	a_{ij}
\end{bmatrix}$, $B=\begin{bmatrix}
b_{ij}
\end{bmatrix}\in \mathcal{M}_n(\mathbb{C}),$ then
\begin{eqnarray}\label{imp}
    w(A)\leq w(B) \quad \textit{whenever\, $0\leq a_{ij}\leq b_{ij}$ for all $i,j.$}
\end{eqnarray}

Clearly  $\sqrt{ \left( \|A_{ij}\|+\|A_{ji}\| \right)^2- \left(\|A_{ij}\| \|A_{ji}\|-w(A_{ji}A_{ij}) \right)} \leq \|A_{ij}\|+\|A_{ji}\|$ for all $i<j$, Hence,
 Theorem \ref{thm2} refines Abu-Omar and Kittaneh's bound \eqref{p8} (via the inequality \eqref{imp}).
\end{remark}

We now move to our next motivation, i.e., the study of Berezin radius inequalities for reproducing kernel Hilbert space operators.
Let $\Omega$ be a non-empty set. A reproducing kernel Hilbert space (RKHS in short) $\mathcal{H}(\Omega)$ is a Hilbert space of complex valued functions on the set $\Omega,$ where point evaluations are continuous (i.e., for each $\lambda \in \Omega,$ the map $E_{\lambda} : \mathcal{H}(\Omega) \to \mathbb{C}$ defined by  $E_{\lambda}(f)=f(\lambda)$ is a bounded linear functional on $\mathcal{H}(\Omega)$. By the Riesz representation theorem for each $\lambda \in \Omega,$ there exists a unique $k_{\lambda} \in \mathcal{H}(\Omega)$ such that $E_{\lambda}(f)=\langle f,k_{\lambda} \rangle$ for all $f \in \mathcal{H}(\Omega).$ The collection of functions  $\{k_\lambda :  \lambda \in \Omega \}$ is the set of all reproducing kernels of $\mathcal{H}(\Omega)$ and $\{\hat{k}_{\lambda}=\frac{k_\lambda}{\|k_\lambda\|} :  \lambda \in \Omega\}$ is the set of all normalized reproducing kernels of $\mathcal{H}(\Omega)$.  For  $A \in  \mathcal{B}(\mathcal{H}(\Omega)),$ the function $\widetilde{A}$, defined on $\Omega$ as $\widetilde{A}(\lambda)=\langle A\hat{k}_{\lambda},\hat{k}_{\lambda} \rangle$,  is called the Berezin symbol (see \cite{BER1974} and \cite{BER1972}) of $A$. We recall that the Berezin set and Berezin radius (or Berezin number) of $A$ are denoted by $\textbf{Ber}(A)$ and $\textbf{ber}(A)$, respectively, and defined (see \cite{KAR_JFA_2006}) as
\[\textbf{Ber}(A):=\left\{\widetilde{A}(\lambda) : \lambda \in \Omega\right\}
 \,\,\text{ and }\,\,
	\textbf{ber}(A):=\sup\left\{|\widetilde{A}(\lambda)| : \lambda \in \Omega\right\}.\] 
   It is clear that $\textbf{Ber}(A) \subseteq W(A),$ so $\textbf{ber}(A)\leq w(A)$. 
Note that $\textbf{ber}(\cdot) : \mathcal{B}(\mathcal{H}(\Omega))\to\mathbb{R}$ does not define a norm, in general. If $\mathcal{H}(\Omega)$ has the ``Ber" property (i.e., for any two operators $A,B \in \mathcal{B}\left(\mathcal{H} (\Omega)\right)$,   $\widetilde{A}\left(
\lambda\right)=\widetilde{B}\left(\lambda\right)$ for all $\lambda \in \Omega$ implies $A=B$), then $\textbf{ber}(\cdot)$ defines a norm.
The Berezin symbol is useful in studying reproducing kernel Hilbert space operators and it has wide application in operator theory. It has been studied in detail for Toeplitz and Hankel operators on Hardy and Bergman spaces. For further information about the Berezin symbol and its applications, we refer the reader to \cite{Bhunia_G}, \cite{K2013}, \cite{KAR_JFA_2006} and \cite{KS_CVTA_2005}. Recently, many mathematicians have studied the
Berezin number inequalities for reproducing kernel Hilbert space operators. 

The Berezin number for operator matrices in $\mathcal{B}(\oplus^n_{i=1}\mathcal{H}(\Omega_i))$ has also been well studied (see e.g. \cite{Bak_CMJ_2018}, \cite{Somdatta}, \cite{SDR_ACTA_2021}), where 
 $\Omega_1,\Omega_2,\ldots,\Omega_n$ are non-empty sets and  $\oplus^n_{i=1}\mathcal{H}(\Omega_i)$ is also a reproducing kernel Hilbert space on $\Omega_1\times \Omega_2\times \cdots \times \Omega_n.$ 
For ${\bf A}=[A_{ij}]$, where $A_{ij}\in \mathcal{B}(\mathcal{H}(\Omega_j), \mathcal{H}(\Omega_i))$, Bakherad showed in \cite[Theorem 2.1]{Bak_CMJ_2018} that
\begin{eqnarray}\label{p-b}
	\textbf{ber}({\bf A}) \leq 	w\left(\begin{bmatrix}
a_{ij}
    \end{bmatrix}_{i,j}^n\right), \textit{ where } a_{ij}=\begin{cases}
        \textbf{ber}(A_{ii})& \textit{ if } i=j,\\
        \|A_{ij}\| & \textit{ if } i\neq j.
    \end{cases}
\end{eqnarray}

Similar to Theorem \ref{thm2}, we can obtain an improvement of \eqref{p-b}. Before presenting this, here we recall, the Berezin norm of $A_{ij}\in \mathcal{B}(\mathcal{H}(\Omega_j), \mathcal{H}(\Omega_i))$ is defined as 
$\|A_{ij}\|_{ber}:=\sup \left\{ \| A_{ij}\hat{k}_{\lambda}\| : \lambda \in \Omega_j\right\}.$
Clearly, $\|A_{ij}\|_{ber}\leq \|A_{ij}\|.$
We prove that
\begin{theorem}\label{thm-ber}
   Let $\mathcal{H}(\Omega_1), \mathcal{H}(\Omega_2), \ldots, \mathcal{H}(\Omega_n)$ be reproducing kernel Hilbert spaces, and let ${\bf A}=[A_{ij}]$ be an $n\times n$ operator matrix with $A_{ij}\in \mathcal{B}(\mathcal{H}(\Omega_j), \mathcal{H}(\Omega_i)).$ Then $\textbf{ber}({\bf A}) \leq 	w\left(\begin{bmatrix}
		a_{ij} 
	\end{bmatrix}_{i,j=1}^n  \right), \textit{ where }$
   \begin{eqnarray*}
       a_{ij}=\begin{cases}
     \textbf{ber}(A_{ii})& \textit{ if } i=j,\\
	    \sqrt{ \left( \|A_{ij}\|_{ber}+\|A_{ji}^*\|_{ber} \right)^2- \left(\|A_{ij}\|_{ber} \|A_{ji}^*\|_{ber}-\textbf{ber}(A_{ji}A_{ij}) \right)} & \textit{ if } i<j,\\
        0 & \textit{ if } i>j.
	\end{cases}
       \end{eqnarray*} 
\end{theorem}

\begin{remark}
    It is easy to show that $a_{ij} \leq \|A_{ij}\|_{ber}+ \|A_{ji}^*\|_{ber} \leq \|A_{ij}\|+ \|A_{ji}\|$ for all $i<j.$ Thus, via the relations \eqref{p9} and \eqref{imp}, we remark that 
      Theorem \ref{thm-ber} refines Bakherad's inequality \eqref{p-b}.
\end{remark}

We now come to our next motivation in this work: as an application of the improved numerical radius estimates, we study spectral radius inequalities for the sums, products and commutators of bounded linear operators in $\mathcal{B}(\mathcal{H})$, which improve and generalize the spectral radius inequalities in \cite{Abu_Stu_2013} and \cite{Kittaneh-AMS}. In \cite[Theorem 1]{Kittaneh-AMS}, Kittaneh showed that if $A_1,A_2,B_1,B_2\in \mathcal{B}(\mathcal{H}),$ then 
    \begin{eqnarray}\label{AMS}
        r(A_1B_1+A_2B_2) &\leq&   \frac{ 1 }{2 } \left( \|B_1A_1\|+ \|B_2A_2\| \right) \notag\\
        && +\frac{ 1 }{2 } \sqrt{ \left( \|B_1A_1\|-\|B_2A_2\|\right)^2 + 4\|B_1A_2\|\|B_2A_1\|}.
    \end{eqnarray}
Further, Abu-Omar and Kittaneh \cite[Theorem 2.2]{Abu_Stu_2013} obtained an improvement of \eqref{AMS}, namely,
    \begin{eqnarray}\label{Stud}
        r(A_1B_1+A_2B_2) &\leq&   \frac{ 1 }{2 } \left(w(B_1A_1)+w(B_2A_2) \right) \notag\\
        && +\frac{ 1 }{2 } \sqrt{ \left(w(B_1A_1)-w(B_2A_2)\right)^2 + 4\|B_1A_2\|\|B_2A_1\|}.
    \end{eqnarray}
In this work, by applying Theorem \ref{thm2}, we can obtain a considerable refinement of \eqref{Stud} with also allowing $A_i \in \mathcal{B}(\mathcal{H}_i, \mathcal{H}_1) $ and $B_i\in \mathcal{B}(\mathcal{H}_1, \mathcal{H}_i)$ for all $i=1,2$ and 
 $\mathcal{H}_1, \mathcal{H}_2$ be arbitrary Hilbert spaces.

\begin{cor}\label{cor-s1}
    Let $\mathcal{H}_1, \mathcal{H}_2$ be complex Hilbert spaces, and let $A_i \in \mathcal{B}(\mathcal{H}_i, \mathcal{H}_1) $ and $B_i\in \mathcal{B}(\mathcal{H}_1, \mathcal{H}_i)$ for all $i=1,2.$
    Then
    \begin{eqnarray*}
        r(A_1B_1+A_2B_2) &\leq &  \frac{ 1 }{2 } \left(w(B_1A_1)+w(B_2A_2) \right) \\
        && +  \frac{ 1 }{2 }  \sqrt{ \left(w(B_1A_1)-w(B_2A_2)\right)^2 + 3\|B_1A_2\|\|B_2A_1\| + w(B_2A_1 B_1A_2)}.
    \end{eqnarray*}
\end{cor}

\begin{remark}
Corollary \ref{cor-s1} refines \eqref{Stud} (as $ w(B_2A_1 B_1A_2)\leq \|B_1A_2\|\|B_2A_1\|$), so also refines \eqref{AMS}.
\end{remark}

A second application of Theorem \ref{thm2} is to provide a refined estimation for the zeros of a complex polynomial 
$p(z)=z^n+a_{n}z^{n-1}+\ldots+a_2z+a_1,$ $n \geq 2,$
where $a_1, a_2, \ldots, a_{n}$ are complex numbers and $a_1\neq 0$. 
We show that 
\begin{theorem} \label{est-poly}
If $\lambda$ is a zero of $p(z),$ then
	\begin{eqnarray*}  
   |\lambda| &\leq& \frac12 \left(|a_n|+\cos \frac{\pi}{n}+ \sqrt{\left(|a_n|-\cos \frac{\pi}{n} \right)^2+\left(1+\sqrt{\sum_{j=1}^{n-1}|a_j|^2} \right)^2-\alpha } \right),
	\end{eqnarray*}
    where $\alpha=\sqrt{\sum_{j=1}^{n-1}|a_j|^2}- \frac{1}{2}\left(|a_{n-1}|+\sqrt{\sum_{j=1}^{n-1}|a_j|^2} \right).$
\end{theorem}

\begin{remark}
   Note that $\alpha >0$ for every non-zero polynomial  $p(z)$. Therefore, Theorem \ref{est-poly} refines 
\begin{equation}\label{abd}
|\lambda| \leq \frac12 \left(|a_n|+\cos \frac{\pi}{n}+ \sqrt{\left(|a_n|-\cos \frac{\pi}{n} \right)^2+\left(1+\sqrt{\sum_{j=1}^{n-1}|a_j|^2} \right)^2 } \right),
\end{equation}
given by Abdurakhmanov \cite{Abdurakhmanov} (also Abu-Omar and Kittaneh \cite{Kittaneh-AFA}).
\end{remark}


\subsection*{Organization of the paper} 
In Section \ref{Sec2}, we develop an inner-product inequality (involving two bounded linear operators), and using this we prove Theorem \ref{thm2} and deduce bounds for the numerical radius of a single bounded linear operator and the product of operators. Using Theorem \ref{thm2}, we also obtain an upper bound for the numerical radius of the Kronecker products $A\otimes B$, where $A\in \mathcal{M}_n(\mathbb{C})$ and $B\in \mathcal{B}(\mathcal{H}),$ which refines a bound in \cite{Khare}.
In Section \ref{S-ber}, similar to Theorem \ref{thm2}, we study Berezin number inequalities for reproducing kernel Hilbert space operators. We prove Theorem \ref{thm-ber} and deduce related results.
In Section \ref{Sec3}, as an application of our numerical radius inequalities for $n \times n$ operator matrices, we obtain several spectral radius inequalities for the sums, products and commutators of bounded linear operators (improving the results in \cite{Abu_Stu_2013}). In particular, we prove Corollary \ref{cor-s1}. 
In Section \ref{Zrors}, by applying the numerical radius inequalities for $2 \times 2$ operator matrices obtained here, we provide a bound for the numerical radius of Frobenius companion matrices.
In particular, we prove Theorem \ref{est-poly}.

\section{Numerical radius inequalities for $n\times n$ operator matrices}\label{Sec2}

In this section we obtain numerical radius inequalities for $n\times n$ operator matrices that strengthen \eqref{p8}, and then deduce an upper bound for the numerical radius of $2\times 2$ operator matrices that refines the bound in \cite[Theorem 2.1]{Paul_IJMMS_2012}. 
Begin by noting the well-known inner-product inequality \cite{Buzano} which is known as Buzano's inequality.



\begin{lemma}[\cite{Buzano}]\label{lem3}
    Let $x,y,z\in \mathcal{H}.$ Then
    $|\langle x,z\rangle \langle z,y\rangle| \leq \frac{1}{2}\left( \|x\|\|y\|+|\langle x,y\rangle|\right) \|z\|^2 .$
\end{lemma}

To prove our main result we need the following inner-product inequality (involving two bounded linear operators).

\begin{lemma}\label{lemma4}
Let $\mathcal{H}_1, \mathcal{H}_2$ be complex Hilbert spaces, and let
 $A\in \mathcal{B}(\mathcal{H}_1,\mathcal{H}_2)$ and $B\in \mathcal{B}(\mathcal{H}_2,\mathcal{H}_1)$. Then 
    $$ |\langle Ax,y\rangle|+|\langle By,x\rangle| \leq \sqrt{ \left( \|A\|+\|B\|\right)^2-\left(\|A\| \|B\|-w(BA) \right)} \, \|x\| \|y\| \quad \forall \, x\in \mathcal{H}_1, \, \forall \, y\in \mathcal{H}_2.$$
\end{lemma}

\begin{proof}
 Using Lemma \ref{lem3}, we get
    \begin{eqnarray*}
     (|\langle Ax,y\rangle|+|\langle By,x\rangle|)^2
        &\leq&  |\langle Ax,y\rangle|^2+|\langle By,x\rangle|^2+ (\|Ax\| \|B^*x\|+|\langle BAx,x\rangle|)\|y\|^2 \\
        &\leq&\left(  \|A\|^2+\|B\|^2+\|A\|\|B\|+ w(BA) \right) \|x\|^2 \|y\|^2,
    \end{eqnarray*}
    as desired.
\end{proof}

With this lemma in hand, we can now refine the Abu-Omar and Kittaneh's inequality \eqref{p8}:

\begin{proof}[Proof of Theorem ~\ref{thm2}]
 Let  $x=  \begin{pmatrix}
	x_1,
	x_2,
	\ldots,
	x_n
\end{pmatrix}^T $ be a unit length vector in $ \oplus_{i=1}^n\mathcal{H}_i $, that is, $  $ $\sum_{i=1}^n\|x_i\|^2=1$.
Then
\begin{eqnarray*}
	|\langle {\bf A}x,x\rangle|
    & = & \left| \sum_{i,j=1}^{n}  \langle A_{ij}x_j,x_i\rangle \right|\\
	& \leq &  \sum_{i,j=1}^{n}  \left| \langle A_{ij}x_j,x_i\rangle \right|
	= \sum_{j=1}^{n}  \left| \langle A_{jj}x_j,x_j\rangle \right|+\sum_{\underset{i\neq j}{i,j=1}}^{n}  \left| \langle A_{ij}x_j,x_i\rangle \right|\\
    &=& \sum_{j=1}^{n}  \left| \langle A_{jj}x_j,x_j\rangle \right|+\sum_{\underset{i< j}{i,j=1}}^{n}  (\left| \langle A_{ij}x_j,x_i\rangle \right|+\left| \langle A_{ji}x_i,x_j\rangle \right|)\\
&\leq & \sum_{j=1}^{n}  w( A_{jj}) \|x_j\|^2+ \sum_{\underset{i< j}{i,j=1}}^{n}  (\left| \langle A_{ij}x_j,x_i\rangle \right|+\left| \langle A_{ji}x_i,x_j\rangle \right|) \\
&\leq & \sum_{j=1}^{n}  w( A_{jj}) \|x_j\|^2+ \sum_{\underset{i< j}{i,j=1}}^{n} a_{ij}\|x_i\| \|x_j\| \quad (\textit{by Lemma \ref{lemma4}})\\
&=& \left \langle \begin{bmatrix}
    a_{ij}
\end{bmatrix}_{i,j=1}^n |x|, |x| \right \rangle,
\end{eqnarray*}
 where $|x|=\begin{pmatrix}
	\| x_1\|,
	 \|x_2\|,
	  \ldots,
	   \|x_n\|
\end{pmatrix}^T\in \mathbb{C}^n$ is a unit vector. This implies 
$|\langle {\bf A}x,x\rangle| \leq w\left(\begin{bmatrix}
    a_{ij}
\end{bmatrix}_{i,j=1}^n \right)$   for all $ x\in \oplus_{i=1}^n\mathcal{H}_i$ with $\|x\|=1,$ and so  $ w({\bf A})\leq w\left(\begin{bmatrix}
    a_{ij}
\end{bmatrix}_{i,j=1}^n \right).$
\end{proof}

From Theorem \ref{thm2}, we deduce an upper bound for the numerical radius of $2 \times 2$ operator matrices.
 Paul and Bag \cite{Paul_IJMMS_2012} (also Abu-Omar and Kittaneh \cite{Abu_LAA_2015}) have
 shown that if $\mathcal{H}_1, \mathcal{H}_2$ are Hilbert spaces, and  $A\in \mathcal{B}(\mathcal{H}_1), ~B\in \mathcal{B}(\mathcal{H}_2, \mathcal{H}_1),~ 
C\in \mathcal{B}(\mathcal{H}_1, \mathcal{H}_2)$ and  $D\in \mathcal{B}(\mathcal{H}_2),$ then
\begin{eqnarray}\label{pb}
        w\left( \begin{bmatrix}
            A & B\\
            C & D
        \end{bmatrix}\right)
        &\leq&  \frac{ 1 }{2 } \left(w(A)+w(D) \right)+  \frac{ 1 }{2 }\sqrt{ \left(w(A)-w(D)\right)^2 + \left(\|B\|+\|C\|\right)^2 }  .
    \end{eqnarray} 
Now Theorem \ref{thm2} yields the following corollary, which refines \eqref{pb}. 

\begin{cor}\label{cor1}
Let $\mathcal{H}_1, \mathcal{H}_2$ be complex Hilbert spaces, and let $A\in \mathcal{B}(\mathcal{H}_1),~ B\in \mathcal{B}(\mathcal{H}_2, \mathcal{H}_1),~ 
C\in \mathcal{B}(\mathcal{H}_1, \mathcal{H}_2)$ and  $D\in \mathcal{B}(\mathcal{H}_2).$ Then
\begin{eqnarray*}
        w\left( \begin{bmatrix}
            A & B\\
            C & D
        \end{bmatrix}\right)
        &\leq&  \frac{1}{2}  \left(w(A)+w(D)\right)\\
        && + \frac{ 1 }{2 } \sqrt{ \left(w(A)-w(D)\right)^2 + \left(\|B\|+\|C\|\right)^2-\left(\|B\|\|C\| -w(CB) \right) } .
    \end{eqnarray*} 
\end{cor}

\begin{proof}
    Setting $n=2$, $A_{11}=A,$ $A_{12}=B$, $A_{21}=C$ and $A_{22}=D$ in Theorem \ref{thm2}, we get
    \begin{eqnarray*}
        w\left( \begin{bmatrix}
            A & B\\
            C & D
        \end{bmatrix}\right)
        \leq  w\left( \begin{bmatrix}
            w(A)& \beta \\
            0& w(D)
        \end{bmatrix}\right)= r\left( \begin{bmatrix}
            w(A)& \frac{1}{2}\beta \\
            \frac{1}{2}\beta & w(D)
        \end{bmatrix}\right) \quad (\textit{by \eqref{p9}}),
        \end{eqnarray*} 
        where $\beta=\sqrt{ \left( \|B\|+\|C\| \right)^2- \left(\|B\| \|C\|-w(CB) \right)}.$ 
\end{proof}

To deduce the numerical radius inequalities for a bounded linear operator and the sum of two operators, we need the following known results.

\begin{lemma}[\cite{Hiz2011}]\label{lem4}
    Let $\mathcal{H}_1, \mathcal{H}_2$ be complex Hilbert spaces, and let $ B\in \mathcal{B}(\mathcal{H}_2, \mathcal{H}_1)$ and $
C\in \mathcal{B}(\mathcal{H}_1, \mathcal{H}_2)$. Then
\begin{enumerate}
    \item $w\left( \begin{bmatrix}
        0&B\\
        C&0
    \end{bmatrix}\right)= \frac{1}{2} \sup_{\theta \in \mathbb{R}}\| e^{i\theta}B+ e^{-i\theta}C^*\|$.

    \item If $\mathcal{H}_1= \mathcal{H}_2$, then $w\left( \begin{bmatrix}
        0&B\\
        B&0
    \end{bmatrix}\right)= w(B)$.
\end{enumerate}
\end{lemma}

    Considering $\mathcal{H}_1=\mathcal{H}_2$ and $A=D=0$,  $B= C$ in Corollary \ref{cor1}, and using Lemma \ref{lem4} (2), we obtain an upper bound for the numerical radius $w(B)$ of $B\in \mathcal{B}(\mathcal{H})$.
    
    \begin{cor}
    If $B\in \mathcal{B}(\mathcal{H})$, then 
    \begin{eqnarray*}
        w(B) &\leq& \sqrt{ \|B\|^2-\frac{1}{4}\left(\|B\|^2-w(B^2) \right)}.
    \end{eqnarray*}
    \end{cor}
    
   This bound is at most the classical bound $\|B\|.$ Therefore, from this inequality we also get the well-known facts: $w(B)=\|B\|$ (e.g. $B$ is normal) implies $\|B\|^2=w(B^2)=\|B^2\|=w^2(B).$

 Next, by considering $A=D=0$ in Corollary \ref{cor1}, and using Lemma \ref{lem4} (1), we obtain a numerical inequality for the sum of two operators.
 
   \begin{cor}\label{cor6}
        Let $\mathcal{H}_1, \mathcal{H}_2$ be complex Hilbert spaces, and let $ B\in \mathcal{B}(\mathcal{H}_2, \mathcal{H}_1)$ and $
C\in \mathcal{B}(\mathcal{H}_1, \mathcal{H}_2)$. Then
\begin{eqnarray*}
  \| B+ C^*\| \leq 2 w\left( \begin{bmatrix}
        0&B\\
        C&0
    \end{bmatrix}\right) \leq \sqrt{\left(\|B\|+\|C\|\right)^2-\left(\|B\|\|C\| -w(CB) \right)}. 
\end{eqnarray*}
\end{cor}

\begin{remark}
    (1) Clearly the right hand side term is at most $ \|B\|+\|C\|,$ so Corollary \ref{cor6} is a nice improvement of the triangle inequality of the operator norm.

    (2) The second inequality in Corollary \ref{cor6} refines and generalizes $w\left( \begin{bmatrix}
        0&B\\
        C&0
    \end{bmatrix}\right) \leq \frac12 (\|B\|+\|C\|)$ for all $B,C\in \mathcal{B}(\mathcal{H}),$ given in \cite[Theorem 2.3]{Hiz2011}.
\end{remark}

Again, from Corollary \ref{cor1}, we deduce an inequality for the sum of two positive bounded operators:

\begin{cor}\label{positive}
    Let $A,B\in \mathcal{B}(\mathcal{H})$ be positive. Then
    \begin{eqnarray*}
        \|A+B\| &\leq& \frac12 \left( \|A\|+\|B\|\right) \\
        &+& \frac{1}{2}\sqrt{\left( \|A\|-\|B\|\right)^2+ \left( \|A^{1-t}B^{1-\alpha} \| + \|B^{\alpha} A^{t}   \| \right)^2-\left( \|A^{1-t}B^{1-\alpha} \| \| B^{\alpha} A^{t}\|   -w(B^{\alpha} A^{}B^{1-\alpha} ) \right) },
    \end{eqnarray*}
   for all $\alpha,t \in [0,1].$ 
\end{cor}

\begin{proof}
   From \cite[Theorem 5.2]{Bhunia-arxiv}, we have
   $\|A+B\| \leq w\left( \begin{bmatrix}
       A& A^{1-t}B^{1-\alpha}\\
       B^{\alpha} A^{t}& B
   \end{bmatrix}\right) .$ This and Corollary \ref{cor1} imply the desired inequality.  
\end{proof}

In particular, for $\alpha=t=\frac12,$ we have
     \begin{eqnarray}\label{posi}
        \|A+B\| &\leq& \frac12 \left( \|A\|+\|B\|\right) 
        + \frac{1}{2}\sqrt{\left( \|A\|-\|B\|\right)^2+4 \|A^{1/2}B^{1/2} \|^2}.
    \end{eqnarray}
This is also proved by Kittaneh \cite{Kittaneh-JOT}. Hence, Corollary \ref{positive} generalizes Kittaneh's inequality \eqref{posi}.

Based on Corollary \eqref{cor-s1}, we obtain a numerical radius inequity for the product of two operators:
\begin{cor}
    Let $A,B\in \mathcal{B}(\mathcal{H}).$ Then
\begin{eqnarray}\label{p--1-}
    w(AB) &\leq& \frac12 w(BA)+\frac14 \sqrt{3\|A\|^2\|B\|^2+w\left(|A|^2|B^*|^2\right)}.
\end{eqnarray}

Moreover, if $AB=BA,$ then
\begin{eqnarray}\label{p--1}
    w(AB) &\leq& \frac12 \sqrt{3\|A\|^2\|B\|^2+w\left(|A|^2|B^*|^2\right)}.
\end{eqnarray}
\end{cor}
\begin{proof}
   The inequality \eqref{p--1} is immediate from \eqref{p--1-}. To show \eqref{p--1-}, consider $A_1=A$, $B_1=B$, $A_2=B^*$ and $B_2=A^*$ in Corollary \ref{cor-s1} to obtain:
    \begin{eqnarray*}
        \| Re(AB) \| &\leq & \frac{1}{2} w(BA)+ \frac14 \sqrt{3\|A\|^2\|B\|^2+w\left(|A|^2|B^*|^2\right)}.
    \end{eqnarray*}
    Replacing $A$ by $e^{i\theta}A$, and then taking the supremum over all $\theta \in \mathbb{R}$, we get $\eqref{p--1-}.$
    \end{proof}

 \begin{remark}
    For $A\in \mathcal{B}(\mathcal{H})$, let $A=U|A|$ be the polar decomposition of $A.$ The generalized Aluthge transform of $A$, denoted as $\Tilde{A}_t$, is defined as  $\Tilde{A}_t=|A|^tU|A|^{1-t}$, $t\in [0,1]$.
    In particular, $\Tilde{A}_0=U^*U^2|A|$, $\Tilde{A}_1=|A|UU^*U=|A|U$ and $\Tilde{A}_{1/2}=|A|^{1/2}U|A|^{1/2}=\Tilde{A}$ (the Aluthge transform of $A$). Replace $A$ and $B$ by $U|A|^{1-t}$ and $|A|^t$, respectively, in \eqref{p--1-} to obtain:
    \begin{eqnarray}
        w(A) &\leq & \frac{1}{2} w(\Tilde{A}_t) +\frac12 \|A\|, \quad \textit{for all } t\in [0,1].
    \end{eqnarray}
    This is also given in \cite{Abu_Stu_2013}.
    Replace $A$ and $B$ by $|A|^t$ and $U|A|^{1-t}$, respectively, in \eqref{p--1-} to obtain:
    \begin{eqnarray}
       w(\Tilde{A}_t)  &\leq & \frac{1}{2} w(A)  +\frac14 \sqrt{3\|A\|^2+w(|A|^{2t} |A^*|^{2(1-t)})},  \quad \textit{for all } t\in [0,1].
    \end{eqnarray}
    In particular, for $t=\frac{1}{2}$, 
     \begin{eqnarray}
      w(\Tilde{A})  &\leq & \frac{1}{2} w(A)  +\frac14 \sqrt{3\|A\|^2+w(|A|^{} |A^*|^{})}.
    \end{eqnarray}
    \end{remark}

 \subsection*{Inequalities for Kronecker products}

A simple consequence of Theorem \ref{thm2} is to provide an upper bound for the numerical radius of the Kronecker product $A \otimes B$, where $A\in \mathcal{M}_n(\mathbb{C})$ and $B\in \mathcal{B}(\mathcal{H}).$

\begin{definition}\label{D12}
The tensor product $\mathcal{H}_1 \otimes \mathcal{H}_2$ of two complex 
Hilbert spaces $\mathcal{H}_1$ and $\mathcal{H}_2$ is defined as the
completion of the inner product space consisting of all elements of the
form $\sum_{i=1}^{n}x_i \otimes y_i$ for $ x_i \in \mathcal{H}_1$ and $ y_i
\in \mathcal{H}_2$, for $n\geq 1$, under the  inner product $\langle x
\otimes y, z \otimes w \rangle :=\langle x , z\rangle \langle y ,
w\rangle.$ 
In particular, $\mathbb{C}^n \otimes \mathcal{H} \cong \mathcal{H}^{\oplus n}$.

The Kronecker product $A \otimes B$ of $A\in
\mathcal{B}(\mathcal{H}_1)$ and $B\in \mathcal{B}(\mathcal{H}_2)$ is defined
as $(A \otimes B)(x \otimes y) := Ax \otimes By$ for $x \otimes y \in
\mathcal{H}_1 \otimes \mathcal{H}_2.$
In particular, if $A=[a_{ij}]\in \mathcal{M}_n(\mathbb{C})$ and $B\in
\mathcal{B}(\mathcal{H})$, the Kronecker product  
$A\otimes B :=  [a_{ij}B]_{i,j=1}^n \in \mathcal{B}( \mathcal{H}^{\oplus n})$
is an $n\times n$ operator matrix.
\end{definition}

In \cite[Theorem 3.4]{double}, Holbrook proved that
\begin{equation}\label{E0-1}
w(A \otimes B) \leq w(A) \| B \|.
\end{equation}

Holbrook proved this in the setting of bounded linear operators $A$ and
$B$ on a Hilbert space $\mathcal{H}$. However, this easily
generalizes to any $A \in \mathcal{B}(\mathcal{H}_1)$ and $ B \in
\mathcal{B}(\mathcal{H}_2)$, where $(\mathcal{H}_1, \mathcal{H}_2)$ denotes an
arbitrary pair of Hilbert spaces; see e.g.\ \cite[Equation
(2)]{BhuniaPMSC}. As a simple consequence of Theorem \ref{thm2},  we can deduce 
an improvement of Holbrooks's bound \eqref{E0-1}, when the complex Hilbert space $\mathcal{H}_1$ is
finite-dimensional. 

\begin{cor}\label{cor3}
    Let $A=[a_{ij}]\in \mathcal{M}_n(\mathbb{C})$ and $B\in \mathcal{B}(\mathcal{H})$. Then $ w(A \otimes B) \leq  w\left(\begin{bmatrix}
		c_{ij} 
	\end{bmatrix}_{i,j=1}^n  \right),$  \textit{where}
    \begin{eqnarray*}
       c_{ij}=  \begin{cases}
       |a_{ii}|w(B) & \textit{ if } i=j,\\
       \sqrt{(|a_{ij}|+|a_{ji}|)^2\|B\|^2-|a_{ij}a_{ji}| \left( \|B\|^2-w(B^2)\right)} & \textit{ if } i<j,\\
       0 & \textit{ if } i>j.
   \end{cases}
   \end{eqnarray*}
        \end{cor}
        
        \begin{proof}
            Since $A\otimes B =  [a_{ij}B]_{i,j=1}^n \in \mathcal{B}( \mathcal{H}^{\oplus n})$
is an $n\times n$ operator matrix,  the desired bound follows from Theorem \ref{thm2}.
        \end{proof}

\begin{remark}
(1) Similar to Remark \ref{rem1}, one can show that the bound in Corollary  \ref{cor3} is at most $w\left([|a_{ij}|]_{i,j=1}^n\right) \|B\|.$ 
Therefore,
Corollary  \ref{cor3} refines Holbrook's bound \eqref{E0-1}, when all
 entries of $A$ are non-negative.

 (2) Also Corollary \ref{cor3} refines the bound $w(A \otimes B) \leq w({C^{\circ}}),$ where ${ C^{\circ}}=[{c}_{ij}]_{i,j=1}^n$ is an $n\times n$ matrix with diagonal entries ${c}_{ii}=|a_{ii}|w(B)$ and off-diagonal entries ${c}_{ij}=|a_{ij}|\|B\|$ for all $1\leq i\neq j \leq n.$ This bound was shown  in \cite[Theorem 1.3]{Khare}.
 \end{remark}

\section{Berezin radius inequalities for $n\times n$ operator matrices}\label{S-ber}

Here we obtain a Berezin number inequality for general $n\times n$ operator matrices, which refines Bakherad's inequality \eqref{p-b}. First, we record an inner-product inequality: 
\begin{lemma}\label{lem-b}
  Let  $A_{12}\in \mathcal{B}(\mathcal{H}(\Omega_2), \mathcal{H}(\Omega_1))$ and  $A_{21}\in \mathcal{B}(\mathcal{H}(\Omega_1), \mathcal{H}(\Omega_2)).$ Then 
  \begin{eqnarray*}
      |\langle A_{12}\hat{k}_{\lambda_2}, \hat{k}_{\lambda_1}\rangle | + |\langle A_{21}\hat{k}_{\lambda_1}, \hat{k}_{\lambda_2}\rangle| 
      \leq \sqrt{ \left(\|A_{12}\|_{ber}+\|A_{21}^*\|_{ber}\right)^2- \left( \|A_{12}\|_{ber}\|A_{21}^*\|_{ber}- \textbf{ber} (A_{21}A_{12})\right) },
  \end{eqnarray*}
  for all $\hat{k}_{\lambda_1}\in  \mathcal{H}(\Omega_1)$ and $ \hat{k}_{\lambda_2}\in \mathcal{H}(\Omega_2).$
\end{lemma}
\begin{proof}
    Similar to Lemma \ref{lemma4}. 
\end{proof}

With this lemma in hand, we now obtain an improvement of Bakherad's inequality \eqref{p-b}:

\begin{proof}[Proof of Theorem ~\ref{thm-ber}]
For $(\lambda_{1}, \lambda_2,\ldots , \lambda_{n})\in \Omega_1\times \Omega_2\times \cdots \times  \Omega_n$, let $\hat{k}_{(\lambda_{1}, \lambda_2, \ldots , \lambda_{n})}=(k_{\lambda_1}, k_{\lambda_2}, \ldots, k_{\lambda_n})$ be the corresponding normalized reproducing kernel of $\oplus_{i=1}^n \mathcal{H}(\Omega_i)$. Then similar to Theorem \ref{thm2} and using Lemma \ref{lem-b}, we obtain that
	\begin{eqnarray*}
		|\langle{\bf A}\hat{k}_{(\lambda_{1}, \ldots , \lambda_{n})}, \hat{k}_{(\lambda_{1}, \ldots , \lambda_{n})}\rangle|
		&=& \left|\sum_{i,j=1}^{n}\langle A_{ij}k_{\lambda_{j}},k_{\lambda_{i}}\rangle\right |\\
	&\leq&\sum_{j=1}^{n}\left|\langle A_{jj}k_{\lambda_{j}},k_{\lambda_{j}}\rangle\right|+\sum_{i,j=1 \atop i<j}^{n}\Big(\left|\langle A_{ij}k_{\lambda_{j}},k_{\lambda_{i}}\rangle\right|+\left|\langle A_{ji}k_{\lambda_{i}},k_{\lambda_{j}}\rangle\right|\Big)\\
		&\leq& \langle {[a_{ij}]_{i,j=1}^n}y,y \rangle, 		
	\end{eqnarray*}
where $y=\begin{pmatrix}
	\|k_{\lambda_{1}}\|,
	\|k_{\lambda_{2}}\|,
	\ldots,
	\|k_{\lambda_{n}}\|
	\end{pmatrix}^T \in \mathbb{C}^n$ with $\|y\|=1$.
Hence, $|\langle {\bf A}\hat{k}_{(\lambda_{1}, \lambda_2, \ldots, \lambda_{n})}, \hat{k}_{(\lambda_{1}, \lambda_2,\ldots , \lambda_{n})}\rangle|\leq w({[a_{ij}]_{i,j=1}^n}).$
Taking the supremum over all $(\lambda_{1}, \lambda_2, \ldots , \lambda_{n})\in \Omega_1\times \Omega_2\times \cdots \times  \Omega_n$, we get
$\textbf{ber}({\bf A})\leq w({[a_{ij}]_{i,j=1}^n}),$ as desired.
\end{proof}

From Theorem \ref{thm-ber}, and similar to Corollary \ref{cor1}, we can deduce a Berezin number inequality for $2\times 2$ operator matrices:

\begin{cor}\label{cor1-b}
Let $\mathcal{H} (\Omega_1), \mathcal{H}(\Omega_2)$ be reproducing kernel Hilbert spaces, and let $A\in \mathcal{B}(\mathcal{H}(\Omega_1)),~ B\in \mathcal{B}(\mathcal{H}(\Omega_2), \mathcal{H}(\Omega_1)),~ 
C\in \mathcal{B}(\mathcal{H}(\Omega_1), \mathcal{H}(\Omega_2))$ and  $D\in \mathcal{B}(\mathcal{H}(\Omega_2)).$ Then
\begin{eqnarray*}
        \textbf{ber}\left( \begin{bmatrix}
            A & B\\
            C & D
        \end{bmatrix}\right)
        &\leq&  \frac{1}{2}  \left(\textbf{ber}(A)+\textbf{ber}(D)\right)\\
        & +& \frac{ 1 }{2 } \sqrt{ \left(\textbf{ber}(A)-\textbf{ber}(D)\right)^2 + \left(\|B\|_{ber}+\|C^*\|_{ber}\right)^2- \left( \|B\|_{ber}\|C^*\|_{ber}- \textbf{ber} (CB)\right) } .
    \end{eqnarray*} 
\end{cor}

Clearly, Corollary \ref{cor1-b} improves the following existing inequality:
\begin{eqnarray}\label{ber-e}
        \textbf{ber}\left( \begin{bmatrix}
            A & B\\
            C & D
        \end{bmatrix}\right)
        &\leq&  \frac{1}{2}  \left(\textbf{ber}(A)+\textbf{ber}(D)\right)
         + \frac{ 1 }{2 } \sqrt{ \left(\textbf{ber}(A)-\textbf{ber}(D)\right)^2 + \left(\|B\|+\|C\|\right)^2 },
    \end{eqnarray} 
     which was given in \cite{Bak_CMJ_2018}. Here we consider an example (Hardy–Hilbert space operators) to show Corollary \ref{cor1-b} is a proper improvement of \eqref{ber-e}. 

\begin{example}
   Recall that (see \cite{Paulsen_book}) the Hardy--Hilbert space of the unit disk  $\mathbb{D}= \{\lambda \in \mathbb{C} : |\lambda |<1 \}$ is denoted by ${H}^2(\mathbb{D})$ and is defined as the Hilbert space of all analytic functions on $\mathbb{D}$ having power series representations with square summable complex coefficients. It is well known that ${H}^2(\mathbb{D})$ is a reproducing kernel Hilbert space. For $\lambda \in \mathbb{D},$ the corresponding reproducing kernel of ${H}^2(\mathbb{D})$ is given by $k_{\lambda}(z)=\sum_{n=0}^{\infty}\bar{\lambda}^n z^n$.
Suppose  ${\bf M}=\begin{bmatrix}
		M_z & 	P_{\mathbb{C}}\\
		P_z & M_{z^2}
	\end{bmatrix} \in \mathcal{B}({H}^2(\mathbb{D})\oplus{H}^2(\mathbb{D}))$ be a $2\times 2$ matrix with the entries $P_{\mathbb{C}}, ~P_z,~ M_z$ and $M_{z^2}$ are respectively defined as  $P_{\mathbb{C}}(f)=\langle f,\phi_0 \rangle,~
	 P_z(f)=\langle f,\phi_1 \rangle \phi_1,~ M_z(f)=\phi_1\cdot f$ and $M_{z^2}(f)=\phi_2\cdot f$,   where $f \in {H}^2(\mathbb{D})$, $\phi_i(z)=z^i$ for all $z\in\mathbb{D}$ and $i=0,1,2$. Then $\textbf{ber}(M_z)=\textbf{ber}(M_{z^2})=1$, $\textbf{ber}(P_{z}  P_{\mathbb{C}} )=0$, $\|P_{\mathbb{C}}\|_{ber} =1$ and $ \|P_z\|_{ber}=\|P_z^*\|_{ber}=\frac12.$ 
Therefore,
	 \begin{eqnarray*}
	 	\textbf{ber}({\bf M})&\leq&  \frac{1}{2}  \left(\textbf{ber}(M_z)+\textbf{ber}(Mz^2_{})\right)\\
        & + & \frac{ 1 }{2 } \sqrt{ \left(\textbf{ber}(M_z)-\textbf{ber}(M_{z^2})\right)^2 + \left(\|P_{\mathbb{C}}\|_{ber}+\|P_z^*\|_{ber}\right)^2- \left( \|P_{\mathbb{C}}\|_{ber}\|P_z^*\|_{ber}- \textbf{ber} (P_zP_{\mathbb{C}})\right) } \\
        &=& \frac{4+\sqrt 7}{4}\\
        &<& 2=  \frac{1}{2}  \left(\textbf{ber}(M_z)+\textbf{ber}(Mz^2_{})\right)
         +  \frac{ 1 }{2 } \sqrt{ \left(\textbf{ber}(M_z)-\textbf{ber}(M_{z^2})\right)^2 + \left(\|P_{\mathbb{C}}\| +\|P_z\|\right)^2 }. 
    \end{eqnarray*}

\end{example}

\section{Spectral radius inequalities for bounded linear operators}\label{Sec3}

As an application of the improved numerical radius inequalities for $n\times n$ operator matrices, we obtain spectral radius inequalities for the sums, products and commutators of bounded linear operators, which improve and generalize the results in \cite{Abu_Stu_2013}. 

Before we prove our results, we need the following basic facts about the spectral radius of bounded linear operators.
It is well-known that $r(A)\leq w(A)$ for every $A\in \mathcal{B}(\mathcal{H}),$ and equality holds if $A$ is normal. Another important property for the spectral radius is the commutative property, i.e.,  $r(AB)=r(BA)$ for every  $A\in \mathcal{B}(\mathcal{H}_1, \mathcal{H}_2)$ and $ B\in \mathcal{B}(\mathcal{H}_2, \mathcal{H}_1).$
Also it is well-known that if $A\in \mathcal{B}(\mathcal{H}_1)$ and $B\in \mathcal{B}(\mathcal{H}_2)$, then $r\left(A\oplus B \right)= \max \{r(A), r(B) \}.$ 
With these results in hand, we can now show the following inequality.

\begin{prop} \label{th3}
Let $\mathcal{H}_1, \mathcal{H}_2, \ldots, \mathcal{H}_n$ be complex Hilbert spaces, and let $A_i \in \mathcal{B}(\mathcal{H}_i, \mathcal{H}_1) $ and $B_i\in \mathcal{B}(\mathcal{H}_1, \mathcal{H}_i)$ for all $i=1,2,\ldots,n.$
 Then $r\left(\sum_{i=1}^n A_iB_i\right) \leq  w \left(\begin{bmatrix}
			a_{ij}
		\end{bmatrix}_{i,j=1}^n \right),$ where 
	\begin{eqnarray*}
		a_{ij}=\begin{cases}
		    w(B_iA_i) & \textit{ if } i=j,\\
            \sqrt{ \left( \|B_iA_j\|+\|B_jA_i\| \right)^2- \left(\|B_iA_j\| \|B_jA_i\|-w(B_jA_iB_iA_j) \right) } & \textit{ if } i<j,\\
            0 & \textit{ if } i>j.
		\end{cases}
	\end{eqnarray*} 
    \end{prop}
    
\begin{proof}
	By letting $A=\begin{bmatrix}
		A_1& A_{2}& \ldots & A_{n} \\
		0& 0& \ldots& 0 \\
		\vdots&\vdots& \ddots & \vdots\\
		0& 0& \ldots& 0\\
	\end{bmatrix}$ and $B=\begin{bmatrix}
	B_1& 0& \ldots & 0 \\
	B_2& 0& \ldots& 0 \\
	\vdots&\vdots& \ddots & \vdots\\
	B_n& 0& \ldots& 0\\
\end{bmatrix} $ in  $\mathcal{B}({\oplus_{i=1}^n}\mathcal{H}_i)$, we obtain 
$$r\left(\sum_{j=1}^n A_jB_j\right) =r\left(\sum_{j=1}^n A_jB_j \oplus 0\right)=r(AB)=r(BA)\leq w(BA).$$
This and Theorem \ref{thm2} imply the desired inequality.
\end{proof}

With this proposition in hand, we now obtain an improvement of Abu-Omar and Kittaneh's inequality \eqref{Stud}:

\begin{proof}[Proof of Corollary ~\ref{cor-s1}]
    From Proposition \ref{th3} (for $n=2$), and using \eqref{p9}, we get
    \begin{eqnarray*}
        r(A_1B_1+A_2B_2) & \leq &   \frac{ 1 }{2 } \left(w(B_1A_1)+w(B_2A_2) \right) \\
        && + \frac{ 1 }{2 } \sqrt{ \left(w(B_1A_1)-w(B_2A_2)\right)^2 + (\|B_1A_2\|+\|B_2A_1\|)^2 -\alpha},
    \end{eqnarray*}
    where $\alpha=\|B_1A_2\|\|B_2A_1\| -w(B_2A_1 B_1A_2) $.
    The desired inequality follows by replacing $A_1$ and $B_1$ by $tA_1$ and $\frac1t B_1$, respectively, and then taking the infimum over $t>0.$
 \end{proof}

 It is well known that if $A,B\in \mathcal{B}(\mathcal{H})$ with $AB=BA$, then $r(A+B)\leq r(A)+r(B)$ and $r(AB)\leq r(A)r(B).$ But, in general, for non-commutative bounded linear operators, the spectral radius is neither subadditive nor submultiplicative. (To see this consider two-dimensional example, with $A=\begin{bmatrix}
     0 &1\\
     0&0
 \end{bmatrix}$ and $B=\begin{bmatrix}
     0 &0\\
     1&0
 \end{bmatrix}$.)
In this regard, we deduce the spectral radius inequalities for the sums, products and commutators of operators, which refine the existing inequalities in \cite{Abu_Stu_2013}.
In \cite[Corollary 2.4]{Abu_Stu_2013}, it was shown that
\begin{eqnarray*}
        r(A+B) \leq   \frac{ 1 }{2 } \left(w(A)+w(B)\right) + \frac12 \sqrt{ \left(w(A)-w(B)\right)^2 + 4 \min \left\{  \|AB\|, \|BA\|\right\}}.
    \end{eqnarray*}

We improve this inequality in the following corollary, where the term  $4 \min \left\{  \|AB\|, \|BA\|\right\}$ is replaced by $\min \left\{  3\|AB\|+w(AB), 3\|BA\|+w(BA) \right\}.$

 \begin{cor}\label{cor-s2}
     Let $A,B\in \mathcal{B}(\mathcal{H}).$ Then
    \begin{eqnarray*}
        r(A+B) &\leq &  \frac{ 1 }{2 } \left(w(A)+w(B)\right)\\
        &&+  \frac{ 1 }{2 } \sqrt{ \left(w(A)-w(B)\right)^2 +  \min \left\{  3\|AB\|+w(AB), 3\|BA\|+w(BA) \right\}} . 
    \end{eqnarray*}
 \end{cor}
 \begin{proof}
     Letting $\mathcal{H}_1=\mathcal{H}_2=\mathcal{H}$, and  $A_1=A,$ $B_2=B$ and $B_1=A_2=I$ in Corollary \ref{cor-s1}, we get
     \begin{eqnarray}\label{s2-1}
        r(A+B) &\leq&   \frac{ 1 }{2 } \left(w(A)+w(B) \right)+\frac12 \sqrt{ \left(w(A)-w(B)\right)^2 + 3 \|BA\|+w(BA) } .
    \end{eqnarray}
     By symmetry (by switching $A$ and $B$), it follows from \eqref{s2-1} that
      \begin{eqnarray}\label{s2-2}
        r(A+B) &\leq&   \frac{ 1 }{2 } \left(w(A)+w(B) \right)+\frac12 \sqrt{ \left(w(A)-w(B)\right)^2 + 3 \|AB\|+w(AB) }.
    \end{eqnarray}
    Hence,  the desired inequality follows by combining \eqref{s2-1} and \eqref{s2-2}.
     \end{proof}

 In \cite[Corollary 2.5]{Abu_Stu_2013}, it was shown that
  \begin{eqnarray*}
         r(AB\pm BA) 
        &\leq & \frac{ 1 }{2 } \left(w(AB)+w(BA)  \right)
         +\frac12 \sqrt{ \left(w(AB)-w(BA)\right)^2 + 4  \|A^2\|\|B^2\|}. 
    \end{eqnarray*}

We improve this inequality in the following corollary, with the term $4  \|A^2\|\|B^2\|$  is replaced by $3  \|A^2\|\|B^2\|+ \min \left\{ w(A^2B^2), w(B^2A^2) \right\}.$
     
 \begin{cor}\label{cor-s3}
        Let $A,B\in \mathcal{B}(\mathcal{H}).$ Then
    \begin{eqnarray*}
         r(AB\pm BA) 
        &\leq & \frac{ 1 }{2 } \left(w(AB)+w(BA)  \right)\\
        && +\frac12 \sqrt{ \left(w(AB)-w(BA)\right)^2 + 3  \|A^2\|\|B^2\|+ \min \left\{ w(A^2B^2), w(B^2A^2) \right\}}. 
    \end{eqnarray*}
     \end{cor}
     
     \begin{proof}
         Letting $\mathcal{H}_1=\mathcal{H}_2=\mathcal{H}$, $A_1=B_2=A,$ $B_1=B$ and $A_2=\pm B$ in Corollary \ref{cor-s1}, we get
         \begin{eqnarray}\label{s3-1}
        r(AB\pm BA) &\leq&   \frac{ 1 }{2 } \left(w(AB)+w(BA) \right) \notag \\
        &&+\frac12 \sqrt{ \left(w(AB)-w(BA)\right)^2 + 3  \|A^2\|\|B^2\|+  w(A^2B^2)} .
    \end{eqnarray}
By symmetry (by switching $A$ and $B$), we also get 
     \begin{eqnarray}\label{s3-2}
        r(AB\pm BA) &\leq&   \frac{ 1 }{2 } \left(w(AB)+w(BA) \right)\notag \\
        &&+\frac12 \sqrt{ \left(w(AB)-w(BA)\right)^2 + 3  \|A^2\|\|B^2\|+  w(B^2A^2)}.
    \end{eqnarray}
    Thus, the desired inequality follows by combining the inequalities \eqref{s3-1} and \eqref{s3-2}.
     \end{proof}

      In \cite[Corollary 2.6]{Abu_Stu_2013}, it was shown that 
    \begin{eqnarray*}\label{1-s4-1}
         r(AB\pm BA) 
        &\leq&   w(AB) + \sqrt{   \min \{  \|A\|\|AB^2\|,  \|B\|\|A^2B\| \} } 
        \end{eqnarray*}
        and
        \begin{eqnarray*}\label{2-s4-1}
         r(AB\pm BA) 
        &\leq&   w(BA) + \sqrt{   \min \{  \|A\|\|B^2A\|, \|B\|\|BA^2\| \} } .
        \end{eqnarray*}

  Similar to the above corollaries, we improve these inequalities in the following result. 

     \begin{cor}\label{cor-s4}
      Let $A,B\in \mathcal{B}(\mathcal{H}).$ Then
    \begin{eqnarray}\label{1-s4}
         r(AB\pm BA) 
        &\leq&   w(AB) +\frac12 \sqrt{  3 \min \{  \|A\|\|AB^2\|,  \|B\|\|A^2B\| \}+  w(A^2B^2) } 
        \end{eqnarray}
        and
        \begin{eqnarray}\label{2-s4}
         r(AB\pm BA) 
        &\leq&   w(BA) +\frac12 \sqrt{  3 \min \{ \|A\|\|B^2A\|, \|B\|\|BA^2\| \}+  w(B^2A^2) } .
        \end{eqnarray}
    \end{cor}

\begin{proof}
 Letting $\mathcal{H}_1=\mathcal{H}_2=\mathcal{H}$, $A_1=I,$ $A_2=B,$ $B_1=AB$ and  $B_2=\pm A$ in Corollary \ref{cor-s1}, we get
\begin{eqnarray}\label{s4-1}
r(AB\pm BA)  &\leq&   w(AB) +\frac12 \sqrt{  3  \|A\|\|AB^2\|+  w(A^2B^2) } .
    \end{eqnarray}

 Again, letting $A_1=AB,$ $A_2=B,$ $B_1=I$ and  $B_2=\pm A$ in Corollary \ref{cor-s1}, we get
        \begin{eqnarray}\label{s4-2}
         r(AB\pm BA) 
        &\leq&   w(AB) +\frac12 \sqrt{  3  \|B\|\|A^2B^\|+  w(A^2B^2). }  \end{eqnarray}
         Combining the inequalities \eqref{s4-1} and \eqref{s4-2}, we obtain the desired inequality \eqref{1-s4}.
         The inequality \eqref{2-s4} follows from \eqref{1-s4} by symmetry (by switching $A$ and $B$).
\end{proof}

In \cite[Corollary 2.7]{Abu_Stu_2013}, is was shown that
\begin{eqnarray*}
        r(AB) 
        &\leq&   \frac{ 1 }{4 } \left(w(AB)+w(BA) \right) \\
        &&+\frac{1}{4}\sqrt{ \left(w(AB)-w(BA)\right)^2 + 4\min \left\{ \|A\|\|BAB\|,  \|B\|\|ABA\| \right\}}.  
    \end{eqnarray*}

Similar to the above corollaries, we refine this inequality in the following corollary.

 \begin{cor}\label{cor-s5}
           Let $A,B\in \mathcal{B}(\mathcal{H}).$ Then
    \begin{eqnarray*}
        r(AB) 
        &\leq&   \frac{ 1 }{4 } \left(w(AB)+w(BA) \right) 
        +\frac{1}{4}\sqrt{ \left(w(AB)-w(BA)\right)^2 + \gamma}, 
    \end{eqnarray*}
    where $\gamma=\min \left\{ 3\|A\|\|BAB\|+  w(ABAB),  3\|B\|\|ABA\|+ w(BABA) \right\}.$
     \end{cor}
     \begin{proof}
          Letting $\mathcal{H}_1=\mathcal{H}_2=\mathcal{H}$, $A_1=\frac12 A,$ $A_2=\frac12 AB,$ $B_1=B$ and  $B_2=I$ in Corollary \ref{cor-s1}, we get
        \begin{eqnarray}\label{s5-1}
        r(AB) 
        \leq  \frac{ 1 }{4 } \left(w(AB)+w(BA) \right) 
        +\frac{1}{4}\sqrt{ \left(w(AB)-w(BA)\right)^2 + 3  \|A\|\|BAB\|+  w(ABAB)} . 
    \end{eqnarray}
    By symmetry (by switching $A$ and $B$), we also get
    \begin{eqnarray}\label{s5-2}
        r(AB) 
        \leq  \frac{ 1 }{4 } \left(w(AB)+w(BA) \right) 
        +\frac{1}{4}\sqrt{ \left(w(AB)-w(BA)\right)^2 + 3  \|B\|\|ABA\|+  w(BABA)} . 
    \end{eqnarray}
    Combining the inequalities \eqref{s5-1} and \eqref{s5-2}, we obtain the desired inequality.
    \end{proof}

To conclude this section, we would like to remark that the spectral radius inequalities obtained here also refine Kittaneh's inequalities in \cite{Kittaneh-AMS}.

\section{Bounds for the roots of an algebraic equation}\label{Zrors}

As another application of the improved numerical radius inequalities for $2 \times 2$ operator matrices, we deduce a refined bound for the zeros of an arbitrary
complex polynomial
\[
p(z)=z^n+a_{n}z^{n-1}+\cdots+a_2z+a_1, \qquad n \geq 2,
\]
where $a_1, a_2, \ldots, a_{n}$ are complex numbers and $a_1\neq 0$.
The Frobenius companion matrix $C(p)$ of $p(z)$ is given by
\[
C(p)=\begin{bmatrix}
-a_{n} \ \ -a_{n-1}\ \ \ldots \ \ -a_2 & -a_1 \\
I_{n-1} & {\bf 0}_{(n-1) \times 1}
\end{bmatrix}.
\]

It is well known  (see \cite[pp.~316]{Horn2}) that the eigenvalues of
$C(p)$ are precisely the zeros of the polynomial $p(z)$. Hence, if $\lambda$ is any zero of $p(z)$, then
\begin{eqnarray}\label{ze}
    |\lambda| \leq r(C(p))\leq w(C(p)).
\end{eqnarray}

In this work, by considering $C(p)$ as a $2 \times 2$ operator matrix 
and  using the numerical radius inequality in Corollary \ref{cor1}, we prove Theorem \ref{est-poly}, which refines \eqref{abd}. For this first we need the following well-known lemmas.

 \begin{lemma}[\cite{Gus}]\label{lemma-L1}
    Let $L_n=\begin{bmatrix}
{\bf 0}_{1 \times (n-1)} & 0 \\
I_{n-1} & {\bf 0}_{(n-1) \times 1}
\end{bmatrix} $. Then $w(L_n)=\cos \frac{\pi}{n+1}.$
 \end{lemma}

\begin{lemma}[\cite{Fujii}]\label{lemma-rank1}
If $a_1, a_2, \ldots, a_{n}\in \mathbb{C}$, then $w\left(\begin{bmatrix} a_1 \ \ a_2 \ \ \ldots \ \ a_n \\ {\bf 0}_{(n-1)
\times n} \end{bmatrix}\right)=\frac12 \left(|a_{1}|+ \sqrt{ \sum_{j=1}^n|a_j|^2}
\right)$.
 \end{lemma}

With these results in hand, we can now improve Abdurakhmanov's bound \eqref{abd}:

\begin{proof}[Proof of Theorem ~\ref{est-poly}]
Consider $C(p)= \begin{bmatrix}
    A&B\\
    C&D
\end{bmatrix}$ as a $2\times 2$ operator matrix, where $A=[-a_n]_{1\times 1}$, $B=[-a_{n-1}, -a_{n-2},\ldots, -a_1]_{1 \times (n-1)},$ $C= [1,0, \ldots,0]^T_{(n-1) \times 1}$ and $D=L_{n-1}$.
Here 
$w(A)=|a_n|$, $w(D)=\cos \frac{\pi}{n}$ (by Lemma \ref{lemma-L1}), $w(CB)=\frac{1}{2}\left(|a_{n-1}|+\sqrt{\sum_{j=1}^{n-1}|a_j|^2} \right)$ (by Lemma \ref{lemma-rank1}),
$\|B\|=\sqrt{\sum_{j=1}^{n-1}|a_j|^2}$ and $\|C\|=1$. Hence,
from Corollary \ref{cor1}, we have
\begin{eqnarray*}
        w\left( C(p)\right)
        &\leq& \frac12 \left(|a_n|+\cos \frac{\pi}{n}+ \sqrt{\left(|a_n|-\cos \frac{\pi}{n} \right)^2+\left(1+\sqrt{\sum_{j=1}^{n-1}|a_j|^2} \right)^2-\alpha } \right),
	\end{eqnarray*}
    where $\alpha=\sqrt{\sum_{j=1}^{n-1}|a_j|^2}- \frac{1}{2}\left(|a_{n-1}|+\sqrt{\sum_{j=1}^{n-1}|a_j|^2} \right).$
    This and \eqref{ze} imply the desired bound.
\end{proof}




\bigskip



\noindent \textbf{Declarations.} Data sharing not applicable to this article as no datasets were generated or analysed during the current study. Author also declares that there is no financial or non-financial interests that are directly or indirectly related to the work submitted for publication.

\bibliographystyle{amsplain}

\end{document}